\documentclass[10pt]{amsart}

\usepackage{sasha_ams}
\usepackage{srcltx}
\usepackage{xcolor}

{}
{}

\newcommand{\Var}{\mathop{\mathrm{Var}}}

\def\P{{\mathbf{P}}}

\def\NP{{\mathbf{NP}}}

\def\PCP{{\mathbf{PCP}}}
\def\BPCP{{\mathbf{BPCP}}}
\def\BGPCP{{\mathbf{BGPCP}}}
\def\GPCP{{\mathbf{GPCP}}}
\def\EP{{\mathbf{EP}}}

\def\TEP{{\mathbf{TEP}}}

\def\HWP{{\mathbf{HWP}}}

\def\CA{{\mathcal{A}}}

\def\Hom{{\mathop{\mathrm{Hom}}}}

\title{The Post Correspondence Problem in groups}

\author[]{Alexei Myasnikov}
\address{Department of Mathematics, Stevens Institute of Technology,
Hoboken, NJ, 07030 USA} \email{amiasnikov@gmail.com}

\author[]{Andrey Nikolaev}
\address{Department of Mathematics, Stevens Institute of Technology,
Hoboken, NJ, 07030 USA} \email{anikolae@stevens.edu}

\author[]{Alexander Ushakov}
\address{Department of Mathematics, Stevens Institute of Technology,
Hoboken, NJ, 07030 USA} \email{aushakov@stevens.edu}

\thanks{The work of the first and third author was partially supported by NSF grant DMS-0914773.}

\begin{document}

\maketitle

\begin{abstract}
We generalize the classical Post correspondence problem ($\PCP_n$) and its non-homogeneous  variation  ($\GPCP_n$) to non-commutative groups and study  the computational complexity of these  new problems.
We observe that $\PCP_n$ is closely related to the equalizer problem in groups, while $\GPCP_n$ is connected to the double twisted conjugacy problem for endomorphisms. Furthermore, it is shown that one of the strongest forms of the word problem in a group $G$ (we call it the {\em hereditary word problem}) can be reduced to $\GPCP_n$ in $G$ in polynomial time.  

The main results are that $\PCP_n$ is decidable in a finitely generated nilpotent group in polynomial time, while $\GPCP_n$ is undecidable in any group containing free non-abelian subgroup (though the argument is very different from the classical case of free semigroups). We show that the double endomorphism twisted conjugacy problem is undecidable in free groups of sufficiently large finite rank. We also consider the bounded $\PCP$ and observe that it is in $\NP$  for any group with $\P$-time decidable word problem, meanwhile it is $\NP$-hard in any group containing free non-abelian subgroup. In particular,  the bounded $\PCP$ is $\NP$-complete in  non-elementary hyperbolic groups  and  non-abelian right angle Artin groups.


\end{abstract}

%\tableofcontents

\section{Introduction}

\subsection{Motivation}

In this paper, following \cite{MNU1}  we continue our research on  non-commutative discrete (combinatorial) optimization. Namely,  we define the Post correspondence problem ($\PCP$) for an arbitrary algebraic structure and then study  this problem  together with its variations for an arbitrary group $G$.   The  purpose of this research is threefold.
Firstly, we approach  $\PCP$ in a very different context, facilitating  a deeper understanding of the nature of $\PCP$ problems in general. Secondly, we try to tackle  several interesting algorithmic problems in group theory that are related to $\PCP$, whose   time complexity is unknown. Thirdly, we hope to   unify several  algebraic  techniques  through the framework of $\PCP$ problems. We refer to \cite{MNU1}  for the initial motivation, the set-up of the problems, and initial facts on non-commutative discrete optimization.

We would like to thank E.~Ventura for his valuable remarks.

\subsection{The Post correspondence problem in algebra}

Let $\CA$ be an arbitrary algebraic structure in a language $L$ (for example, a  semigroup, a group,  or a ring). The Post correspondence problem for $\CA$ (abbreviated as $\PCP(\CA)$) asks to decide when given two tuples of equal length $u = (u_1, \ldots,u_n)$ and $v = (v_1, \ldots,v_n)$ of elements of $\CA$  if there is a term $t(x_1, \ldots,x_n)$  in the language $L$ such that $t(u_1, \ldots,u_n) = t(v_1, \ldots,v_n)$ in $\CA$. In 1946 Post introduced this problem in the case of free monoids (free semigroups)  and proved that it is undecidable \cite{Post:1946}. 
Since then $\PCP$ took its prominent place in the theory of algorithms and theoretical computer science.

There are some interesting variations of this problem especially in the case of semigroups and groups, which we discuss in detail in Section \ref{se:PCP}.  Here we mention only one,  designed specifically for  (semi)groups,  to which we refer as a general or a non-homogeneous Post correspondence problem ($\GPCP$). In this case the terms $t$ are just words in a fixed alphabet $X$ (or $X \cup X^{-1}$ in the case of groups), and the problem is to decide when given two tuples $u$ and $v$  of elements in a (semi)group $S$ as above and two extra elements $a, b \in S$ if there is a term $t(x_1, \ldots,x_n)$ such that $at(u_1, \ldots,u_n) = b t(v_1, \ldots,v_n)$ in $S$. 

Above we described a {\em decision} version of $\PCP$ and $\GPCP$ in a semigroup (or a group) which requires to check if there exists a   term   $w$,   called a {\em solution}, for a given instance of the problem. The {\em search} variation of the problem  asks to find a solution (if it exists) for a given instance. Even more interesting problem is to describe all solutions to the given instance of the problem. We will have more to say about this in  due course.

\subsection{Algebraic meaning of $\PCP$ and $\GPCP$ in groups}

Some connections between Post correspondence problems and classical questions in groups are known. We mention some of them here and refer to Sections \ref{se:connections} and \ref{se:GPCP} for details.

The standard  (homogeneous) $\PCP$ in groups is closely related to the problem of finding  the equalizer $E(\phi,\psi)$ of two group homomorphisms $\phi,\psi : H \to G$.   This equalizer  is defined as $E(\phi,\psi) = \{w \in H \mid \phi(w) = \psi(w)\}$.  In particular, $\PCP$ in a group $G$ is  the same as to decide if the equalizer of a given pair of  homomorphisms $\phi, \psi \in Hom(H,G)$, where $H$ is a free group of finite rank in the variety $\Var(G)$ generated by $G$,  is trivial or not (see Section \ref{se:connections} for details).  Indeed, in this case every tuple $u = (u_1, \ldots,u_n)$ of elements of $G$ gives rise to a homomorphism $\phi_u$ from a free group $H$ with basis $x_1, \ldots,x_n$ in the variety $\Var(G)$ such that $\phi_u(x_1) =  u_1, \ldots, \phi_u(x_n) = u_n$, and vice versa. The equalizer $E(\phi_u,\phi_v)$ describes all  solutions $w$ for the instance $u,v$.

It seems that the general Post correspondence problem $\GPCP$ for groups is even more interesting than the  standard  $\PCP$. Indeed, first of all $\GPCP$  is right in the  midst of the endomorphic double twisted conjugacy problem  in groups, which is one of the  more difficult and less studied  group theoretic conjugacy-type  problems.  In fact, it is shown in Section~\ref{sec:twisted} (Proposition \ref{pr:reduction-GPCP-2}) that the double endomorphism twisted conjugacy problem in a relatively free group in $\Var(G)$ is equivalent to $\GPCP(G)$, and, in general, the double endomorphism twisted conjugacy problem in $G$ $\P$-time reduces to $\GPCP(G)$. Furthermore, we prove in Section \ref{se:GPCP} that $\GPCP$ in a given group $G$ is intimately related to the word problem in $G$. Namely we show that the hereditary word problem ($\HWP$) in $G$ can be reduced in polynomial time to $\GPCP$ in $G$. Here  $\HWP$ in $G$  asks to decide when given an element $w \in G$ and a finite subset $R \subseteq G$ if $w = 1$ in the quotient  $H = G/\langle R\rangle_G$ of the group $G$ by the normal subgroup $\langle R\rangle_G$ generated by $R$. Therefore, if $\GPCP$ is decidable in $G$ then there is a {\em uniform} algorithm to decide the word problem in every finitely presented (relative to $G$) quotient of $G$. Further, since decidability of $\GPCP$ in $G$ is inherited by all subgroups of $G$ it implies the uniform decidability of $\HWP$ in every subgroup of $G$ (even every section of $G$).   Thus, a decision algorithm for $\GPCP$ in $G$  is very powerful and it  gives a lot of information about the group $G$. Notice that finitely generated abelian and nilpotent   groups have  decidable $\HWP$. 

\subsection{Results}

In Section \ref{se:GPCP} we show that $\GPCP$ is undecidable in every non-abelian free group, as well as in every group containing free non-abelian subgroups. In particular,  $\GPCP$ is undecidable in the following groups: non-elementary hyperbolic, non-abelian right angled Artin,  braid groups $B_n$ ($n \geq 3$), non-solvable defined by a single relator (thus all one-relator groups with  more than two generators), etc.  A similar argument shows that the bounded Post correspondence  problem in all groups mentioned above is $\NP$-complete. Here in the bounded version of $\GPCP$ one is looking only for solutions (the words $t(x_1, \ldots,x_n)$) whose  length is bounded by a given number.   We emphasize that the argument used to prove the undecidability results here has nothing to do with the original argument of undecidability of $\PCP$ or $\GPCP$ in free non-commutative semigroups, even though all the groups mentioned above contain such semigroups. In fact, it is still unclear if the $\PCP$ in a free semigroup can be reduced to $\PCP$ in a free non-abelian group. Furthermore, it is still one of the most intriguing open problems whether $\PCP$ in a free non-abelian group is decidable or not. 

As a corollary of the undecidability of $\GPCP$ in free non-abelian groups we show that the double endomorphism twisted conjugacy problem in free groups $F_n$ of rank $n \geq 32$ is undecidable. Whether  the double endomorphism twisted conjugacy problem  is decidable or not in free non-abelian groups of smaller rank remains to be seen.

We also show that free solvable groups $S_{m,n}$ of class $m\geq 3$ and sufficiently high rank $n$ have undecidable double endomorphism twisted conjugacy problem, as well as $\GPCP$. This result is based on examples of finitely presented solvable groups with undecidable word problem constructed by Kharlampovich in \cite{Kharlampovich:1981}.

In the opposite direction we show in Section \ref{se:nilpotent} that $\PCP$ is decidable in polynomial time  in every finitely generated nilpotent group $G$. This is the best known positive result up to date on $\PCP$ in groups. 

%In Section \ref{se:problems} we discuss some open problems on this subject.

\section{Post correspondence problems}
\label{se:PCP}

\subsection{The classical Post correspondence problem} \label{se:semigroups}

Let $A$ be a finite alphabet with $|A| \geq 2$. Denote by $A^\ast$ the free monoid with basis $A$  viewed as the set of all words in $A$ with concatenation as the multiplication.  Let $X$ be an infinite countable set of variables and $X^\ast$ the corresponding free monoid.

\medskip
\noindent{\bf The classical  Post correspondence problem ($\PCP$) in $A^*$:}
Given a finite set of pairs $(g_1,h_1), \ldots , (g_n,h_n)$ of elements of $A^\ast$  determine if there is a non-empty word $w(x_1, \ldots,x_n) \in X^\ast$  such that  $w(g_1, \ldots,g_n) = w(h_1, \ldots,h_n) $ in $A^\ast$.

\medskip
Post showed in  \cite{Post:1946} that the problem is undecidable (see \cite{Sipser:2005} for a simpler proof). 

Nowadays there are several variations of  $\PCP$ in $A^*$, the following {\em restricted} version   is the most typical. 

\medskip
\noindent{\bf $\PCP_n$ in $A^*$:}
Let $n$ be a fixed positive integer. Given a finite sequence of pairs $(g_1,h_1), \ldots , (g_m,h_m)$ of $G$, where $m \leq n,$ determine if there is a non-empty word $w(x_1, \ldots,x_m) \in X^*$   such that  $w(g_1, \ldots,g_m) = w(h_1, \ldots,h_m)$ in $A^*$.

\medskip
Breaking $\PCP$ into a collection of the restricted problems $\PCP_n$  makes the boundary between decidable and undecidable more clear: $\PCP_n$ in $A^*$  is decidable for $n \leq 3$, and undecidable for $n\geq 7$, see \cite{EKR,Halava00binary(generalized),Matiyasevich:1996}.

Another version of interest is the general $\GPCP$ in the free monoid $A^*$, in which case an input to $\PCP$  contains a sequence of pairs $(g_1,h_1), \ldots , (g_n,h_n)$ as above and also four elements $a,b,c,d \in A^*$, while  the task is to find a word $w(x_1, \ldots,x_n) \in X^*$   such that  $aw(g_1, \ldots,g_n)b = cw(h_1, \ldots,h_n)d$ in $A^*$. This problem is also undecidable in $A^*$. 

There are {\em marked} variations of the $\PCP_n$  in $A^*$, in which case for each pair $(g_i,h_i)$ in the instance the initial letters in $g_i$ and $h_i$ are not equal.  These problems are known to be decidable \cite{Halava-Hirvensalo-deWolf:1999}.  We refer to a paper \cite{Halava-Harju:2001} for
some recent developments on the Post correspondence problem in free semigroups.

Finishing our short survey of known results we would like to mention that  $\PCP$ is undecidable in a free non-abelian semigroup as well (the same argument as for free monoids). Hence  semigroup version of $\PCP$ is also undecidable in semigroups that contain free non-abelian subsemigroups, in particular, in groups containing free non-abelian subgroups, or solvable not virtually nilpotent groups (they  contain free non-abelian subsemigroups).  

In what follows we focus only on the group theoretic versions of the Post corresponding problems $\PCP$ and $\GPCP$ in groups, which is different from the original semigroup version since one has to take inversion of elements into account. 

\subsection{The Post correspondence problem in groups}

Throughout the whole paper we use the following notation:  $G$ is an arbitrary  fixed group generated by a finite set $A$,   
$F(X)$ a free group with basis $X = \{x_1, \ldots, x_n\}$.  We view elements of $F(X)$ as reduced words in $X \cup X^{-1}$. 
Sometimes we denote $F(X)$ as $F(x_1, \ldots,x_n)$, or simply as $F_n$. 

As we mentioned earlier,  the  group theoretic version of the Post corresponding problem involves terms (words) with inversion. 

\medskip
\noindent{\bf The Post correspondence problem ($\PCP$) in a group  $G$:}
Given a finite set of pairs $(g_1,h_1), \ldots , (g_n,h_n)$ of elements of $G$ determine if there is a word $w(x_1, \ldots,x_n) \in F(x_1, \ldots,x_n)$, which is not an identity of $G$,   such that  $w(g_1, \ldots,g_n) = w(h_1, \ldots,h_n)$ in $G$.

\medskip
Several comments are in order here. Recall that an identity on $G$ is a word $w(x_1, \ldots,x_n)$ such that $w(g_1, \ldots,g_n) =  1$ in $G$ for any $g_1, \ldots, g_n \in G$. If the group $G$ does not have non-trivial identities then the requirement that $w$ is not an identity becomes the same as in the original  Post formulation  that $w$ is non-empty.  Meanwhile, any non-trivial identity $w(x_1, \ldots,x_n)$ in $G$ gives a solution to any instance of $\PCP$ in $G$, which is not very interesting. Sometimes we refer to words $w$ which are identities in $G$ as to   {\em trivial} solutions of $\PCP$ in $G$, while the solutions which are not identities in $G$  are termed  {\em non-trivial}. In this regard $\PCP(G)$ asks to find a non-trivial solution to $\PCP$ in $G$.

In the sequel by $\PCP$ for  a group $G$ we always, if not said otherwise,  understand the group theoretic (not the semigroup one) version of $\PCP$ stated above. By definition $\PCP(G)$ depends on the given generating set of $G$, however it is easy to see that $\PCP(G)$ for different finite generating sets are  polynomial time equivalent to each other, i.e., each one reduces to the other in polynomial time. Since in all our  considerations  the generating sets are finite we omit them from notation and write  $\PCP(G)$.

Similar to the classical case one can define  the restricted version $\PCP_n$ for a group $G$, in which case the number of pars in each instance of $\PCP_n$ is bounded by $n$,   and the general one $\GPCP$ (or $\GPCP_n$), where there are some constants involved.  Since the general version is of crucial interest for us we state it precisely.

\medskip
\noindent
{\bf The general Post correspondence problem ($\GPCP$) in  a group $G$:}  given a finite sequence  of pairs $(g_1,h_1), \ldots , (g_n,h_n)$ and two  pairs $(a_1,b_1)$ and $(a_2,b_2)$ of elements of $G$ (called the {\em constants} of the instance)   determine if there is a word $w(x_1, \ldots,x_n) \in F(x_1, \ldots,x_n)$ such that  $a_1w(g_1, \ldots,g_n)b_1 = a_2w(h_1, \ldots,h_n)b_2$ in $G$.

\medskip
Two lemmas are due here. 

\begin{lemma}
For any group $G$ $\GPCP(G)$ is linear time equivalent to the restriction of $\GPCP(G)$ where the constants $b_1, b_2,a_2$ are all equal to 1. 
\end{lemma}
\begin{proof}
Indeed, in the notation above notice that $a_1w(g_1, \ldots,g_n)b_1 = a_2w(h_1, \ldots,h_n)b_2$ in $G$ if and only if 
$$a_2^{-1}a_1w(g_1, \ldots,g_n)b_1b_2^{-1} = w(h_1, \ldots,h_n),
$$
 so $\GPCP$  in $G$ is equivalent to $\GPCP$ with $a_2 = 1, b_2 = 1$. Moreover,  
 $$aw(g_1, \ldots,g_n)b = w(h_1, \ldots,h_n)$$
  in $G$ if and only if 
 $$abb^{-1}w(g_1, \ldots,g_n)b = w(h_1, \ldots,h_n),$$
  i.e., 
  $$abw(g_1^b, \ldots,g_n^b) = w(h_1, \ldots,h_n).$$
   Hence $\GPCP(G)$  is linear time equivalent to $\GPCP(G)$ with $b_1= a_2 = b_2 = 1$, as claimed.  \end{proof}
  From now on we often assume that in $\GPCP$ each instance  has the constants $b_1, b_2,a_2$ are all equal to 1, in which case we denote $a_1$ by $a$ and term it the {\em constant} of the instance.  
    
     \begin{lemma} \label{le:reduction-GPCP}
   For any group $G$ and for any instance $(g_1,h_1), \ldots , (g_n,h_n), a$ of $\GPCP(G)$ all  solutions $w$ to this instance can be described as $w = w_0u$, where $w_0$ is a particular fixed solution to this instance and $u$ is an arbitrary (perhaps, trivial) solution to $\PCP(G)$ for the instance $(g_1,h_1), \ldots , (g_n,h_n)$.
  \end{lemma}
   \begin{proof}
   Suppose $w_0$ is a particular fixed solution to $\GPCP(G)$ for the instance $(g_1,h_1), \ldots , (g_n,h_n), a$, so $aw_0(g_1, \ldots,g_n)  = w_0(h_1, \ldots,h_n)$. If $w$ is an arbitrary solution to the same instance in $G$ then  $aw(g_1, \ldots,g_n)  = w(h_1, \ldots,h_n)$, so 
   $$w_0^{-1}(g_1, \ldots,g_n)w(g_1, \ldots,g_n) = w_0^{-1}(h_1, \ldots,h_n)w(h_1, \ldots,h_n),$$
   hence  $u = w_0^{-1}w$ solves  $\PCP(G)$ for the instance $(g_1,h_1), \ldots , (g_n,h_n)$. Therefore, $w = w_0u$ as claimed.
 \end{proof}
 
 Lemma \ref{le:reduction-GPCP} shows that to get all solutions of $\GPCP$ in $G$ for a given instance  one needs only to find a particular solution of $\GPCP(G)$ and all solutions of $\PCP(G)$ for the same instance. In view of this we  sometimes refer   to  $\GPCP$ as the  {\em non-homogeneous} $\PCP$, and to $\PCP$  as to the  {\em homogeneous} one.  

As usual in discrete optimization there  are several other standard variations of  $\PCP$ problems: {\em bounded, search}, and {\em optimal}. We mention them briefly now  and refer to \cite{MNU1} for a thorough discussion of these types of problems in groups.  The {\em bounded} version of $\PCP$ (or $\GPCP$) requires that the word $w$ in question should be of length bounded from above by a given number $M$.  
We denote these versions by $\BPCP(G)$ or $\BGPCP(G)$. The {\em search} variation of $\PCP$ (or $\GPCP$) asks to find a    word $w$ that gives  a non-trivial solution to a given instance of the problem (if such a solution exists). The {\em optimization} version of $\PCP$ (or $\GPCP$) is a variation of the search problem, when one is asked to find a solution that satisfies some ``optimal'' conditions. In our case, if not said otherwise, the optimal condition is to find a shortest possible word $w$ which is a solution to the given instance of the problem.

\section{Connections to group theory}
\label{se:connections}

\subsection{$\PCP_n$ and the equalizer problem}

Let as above $G$ be a fixed arbitrary group with a finite generating set $A$, $F_n = F(x_1, \ldots,x_n)$  a free group with basis $X= \{x_1, \ldots,x_n\}$.

An $n$-tuple of elements $g= (g_1, \ldots,g_n) \in G^n$ gives  a homomorphism $\phi_g:F_n \to G$ where $\phi_g(x_1) = g_1, \ldots, \phi_g(x_n) = g_n$. And vice versa, every homomorphism $F_n \to G$ gives a tuple as above. In this sense each instance 
$(u_1,v_1), \ldots, (u_n,v_n)$ of  $\PCP(G)$ can be uniquely described by a pair of homomorphisms $\phi_u,\phi_v:F_n\to G$, where $u = (u_1, \ldots,u_n), v = (v_1, \ldots, v_n)$.  In this case we refer to such a pair of homomorphisms  as an instance of $\PCP$ in $G$.

Now given groups $H,G$ and two homomorphism $\phi, \psi \in \Hom(H,G)$  one can define the   equalizer $E(\phi,\psi)$ of $\phi,\psi$  as
\begin{equation}\label{eq:equalizer}
E(\phi,\psi) = \{w \in H \mid w^\phi = w^\psi\},
\end{equation}
which is obviously a subgroup of $H$.
If $G$ does not have non-trivial identities then all non-trivial words from  $E(\phi,\psi)$ give all solutions to $\PCP$    in $G$ for a given instance $(\phi, \psi) \in \Hom(F_n,G)$. However, if $G$ has non-trivial identities then some words from $E(\phi,\psi)$ are identities which are not solutions to $\PCP(G)$.   To accommodate all the cases at once we suggest to replace the free group $F_n$ above by the  free group $F_{G,n}$ in the variety $\Var(G)$ of rank $n$ with basis $\{x_1, \ldots, x_n\}$. Then similar to the above every tuple $u \in G^n$ gives rise to a homomorphism $\phi_u:F_{G,n} \to G$, where $\phi(x_1) = u_1, \ldots, \phi(x_n) = u_n$, and non-trivial elements of the equalizer $E(\phi_u,\phi_v)$ describe all solutions of $\PCP(G)$ for the instance $u,v \in G^n$.  This connects $\PCP_n$ in $G$ with the equalizers of homomorphisms  from $\Hom(F_{G,n},G)$. 

There  are two general algorithmic problems in groups concerning equalizers. 

\medskip
\noindent
{\bf  The triviality of the equalizer problem $ (\TEP(H,G))$  for groups $H,G$:} Given two homomorphisms $\phi, \psi \in \Hom(H,G)$ decide if the subgroup $E(\phi,\psi)$ in $H$ is trivial or not.

\medskip
\noindent
{\bf  The equalizer problem $ (\EP(H,G))$  for groups $H,G$:} Given two homomorphisms $\phi, \psi \in \Hom(H,G)$ find the equalizer $\EP(H,G)$. In particular, if  $\EP(H,G)$ is finitely generated then find a finite generating set of $E(\phi,\psi)$.

\medskip
The formulation above needs some explanation on how we mean ``to find''  a subgroup in a group. If the subgroup is finitely generated then ``to find'' usually means to list a finite set of generators. It might happen that the subgroup is not finitely generated, but allows a finite set of generators as a normal subgroup, or as a module under some action. In  this case to solve $ \EP(H,G)$ one has to list  a finite set of these generators of $ \EP(H,G)$. In this paper we consider equalizers of homomorphisms of finitely generated nilpotent groups, so in this event they  are finitely generated and the problem of describing equalizers becomes well-stated. 

 Equalizers $E(\phi,\psi)$ were studied before, but  mostly in the case when $H = G$ and $\phi, \psi$ are automorphisms of $G$. 
There are few results on equalizers of endomorphisms in groups.
Goldstein and Turner have proved in \cite{Goldstein-Turner:1986} that the equalizer
of two endomorphisms of $F_n$ is a finitely generated subgroup in the case one of the
two maps is injective. However, is it not known whether there is an algorithm to decide if the equalizer of two endomorphisms in a free group $F_n$ is trivial or not.
Ciobanu, Martino and Ventura showed that generically equalizers of endomorphisms in free groups are trivial \cite{Ciobanu08thegeneric}, so on most  inputs in a free non-abelian group $F$  $\PCP(F)$ does not have a solution, in this sense $\PCP(F)$ is generically decidable.

We summarize the discussion above in the following easy lemma.  

\begin{lemma}
Let $G$ be a group. Then the following holds for any natural $n>0$:
\begin{itemize}
\item [1)] $\PCP_n(G)$ is equivalent (being just a reformulation) to $\TEP$ for homomorphisms from $Hom(F_{G,n},G)$.
\item [2)] Finding all solutions  for a given instance of $\PCP_n(G)$  is equivalent (being just a reformulation) to $\EP(F_{G,n},G)$ for the same instance. 
\end{itemize}
\end{lemma}

\subsection{$\GPCP$ and the double twisted conjugacy}
\label{sec:twisted}

Let $\phi, \psi$ be two fixed automorphisms of a group $G$.  Two elements $u, v \in G$ are termed  {\em $(\phi,\psi)$-double-twisted conjugate} if there is an element $w \in G$ such that  $uw^\phi = w^\psi v$. In particular, when $\psi = 1$ then $u$ and $v$ are called {\em $\phi$-twisted conjugate}, while in the case $\phi = \psi = 1$  $u$ and $v$ are just usual conjugates of each other. The twisted (or double twisted) conjugacy problem in $G$ is to decide whether or not two  given elements $u, v \in G$ are twisted (double twisted) conjugate  in $G$  for a fixed pair of automorphisms $\phi, \psi \in \mathop{\mathrm{Aut}}(G)$. Observe, that since $\psi$ has the inverse the $(\phi,\psi)$-double-twisted conjugacy problem reduces to  $\phi\psi^{-1}$-twisted conjugacy problem, so in the case of automorphisms it is sufficient to consider only twisted conjugacy problem. This  problem is much studied in groups, we refer to
\cite{Ventura-Romankov:2009, Romankov:2010, Romankov:2011, BMMV, Troitsky, Fel'shtyn,Fel'shtyn-Leonov-Troitsky} for some recent results.  

Much stronger versions of the problems above appear when one replaces automorphisms by arbitrary endomorphisms $\phi, \psi  \in End(G)$.  Not much is known about  double twisted conjugacy problem in groups with respect to endomorphisms.    

The next statement (which follows from the discussion  above) relates the  double-twisted conjugacy problem for endomorphisms to the general Post correspondence problem. 

\begin{proposition} \label{pr:reduction-GPCP-2}
Let $G$ be a group generated by a finite set $A = \{a_1, \ldots,a_n\}$. Then the following holds:
\begin{itemize}
\item [1)] The double-twisted conjugacy problem for endomorphisms in $G$ is linear time reducible to $\GPCP_n(G)$.
\item [2)]  If $G$ is relatively free with basis $A$ then the double-twisted conjugacy problem for endomorphisms in $G$ is linear time equivalent  to $\GPCP_n(G)$.
\end{itemize}
\end{proposition}

\section{The hereditary word problem and  $\GPCP$}
\label{se:GPCP}
It is easy to see that decidability of $\PCP_n$ or $\GPCP_n$ in a group $G$ has some implications for the word problem in $G$. 
Indeed, an element $g$ is equal to 1 in $G$ if and only if $\GPCP_1$ is decidable in $G$ for the instance consisting of a single pair  $(1,1)$ and the constant $g$.  Similarly, if $G$ is torsion-free then $g  = 1$ in $G$ if and only if $\PCP$ is decidable in $G$ for the instance pair $(g,1)$.   In this section we show that the whole lot of word problems in the quotients of $G$ is reducible to $\GPCP$ in $G$.

Let $G$ be a group generated by a finite set $A$. For a subset $R\subseteq G$ by $\langle R\rangle_G$ we denote the normal closure of $R$ in $G$.

\medskip
\noindent
{\bf   The hereditary word problem ($\HWP(G)$) in $G$:} Given a finite set $R\cup\{w\}$ of words in the alphabet $A\cup A^{-1}$, decide  whether or not  $w$ is  trivial in  the quotient  $G/\langle R\rangle_G$. 

\medskip
Note that  this problem can also be stated as  the uniform membership problem to normal finitely generated subgroups of $G$. Observe also that $\HWP(G)$ requires a {\em uniform}  algorithm for the word problems in the quotients $G/\langle R\rangle_G$.  

It seems that groups with decidable $\HWP$ are rare. Notice that the hereditary word problem is decidable in finitely generated abelian or nilpotent groups.

\begin{proposition}\label{prop:WPtoPCP}
Let $G$ be a finitely generated group. Then the  hereditary word problem in $G$ $\P$-time reduces to $\GPCP(G)$.
\end{proposition}
\begin{proof}
Let $A$ be a finite generating set of $G$.  Suppose  $R$ is a finite set of elements of $G$, represented by words in $A\cup A^{-1}$.
%To simplify notation we can assume that $A$ and $R$ are closed under inversion.
Denote $H=G/\langle R\rangle_G$.
 Put 
 $$D_R=\{(a,a^{-1})\mid a\in A\}\cup\{(a,a^{-1})\mid a\in A\}\cup \{(r,1)\mid r \in R\}\cup \{(r^{-1},1)\mid r \in R\}.
$$

\noindent{\bf Claim 1.} Let $w$ be a word $w\in (A\cup A^{-1})^\ast$. Then $w=_H1$ if and only if there is a finite sequence of pairs
$(u_1,v_1),\ldots,(u_k,v_k)\in D_R$ such that
\begin{equation} \label{eq:u-v}
v_n(\cdots (v_2(v_1 wu_1)u_2)\cdots )u_n=_G 1.
\end{equation}
Indeed, if (\ref{eq:u-v}) holds then 
$$
w=_G v_1^{-1} \ldots v_{n-1}^{-1}(v_n^{-1}u_n^{-1})u_{n-1} ^{-1}\ldots u_1^{-1} =_H 1
$$
since for every pair $(u,v) \in D_R$ one has $uv = 1$ in $H$. 

 To show the converse, suppose $w =_H 1$, i.e., $w \in \langle R\rangle_G$. In this case 
\begin{equation}\label{eq:wp}
w=_Gw_1r_1w_2\ldots w_{m}r_mw_{m+1}
\end{equation} with $r_i\in R, w_i\in A^\ast$ and $w_1w_2\ldots w_{m+1}=_G1$.
Rewriting (\ref{eq:wp}) one gets 
\begin{equation}\label{eq:wp2}
r_1^{-1} \cdot  w_1^{-1} \cdot w\cdot w_1 \cdot 1 =_G w_2r_2w_3\ldots w_{m}r_mw_{m+1}w_1.
\end{equation}
 Notice that the product on the left is in the form required in (\ref{eq:u-v}), and the product on the right is in the form required in  (\ref{eq:wp}). Now the result follows by induction  on  $m$. This proves the  claim.

\medskip
\noindent{\bf Claim 2.} Let $R\subseteq (A\cup A^{-1})^*$ be a finite set and $w \in (A\cup A^{-1})^*$. Then $\GPCP(G)$ has a solution for the instance  
$\hat{D}_R = \{(u,v^{-1})\mid(u,v)\in D_R\}$ with the constant $w$ if and only if $w = 1$ in $H$.

\medskip 
Indeed,  a sequence
\begin{equation}\label{eq:pcp_witness}
(u_1,v_1^{-1}),\ldots ,(u_M,v_M^{-1}) \in \hat{D}_R
\end{equation} 
  gives a solution to $\GPCP(G)$ for  the instance $\hat{D}_R$ with the constant $w$  if and only if
\[
wu_1u_2\cdots u_M=_Gv^{-1}_1v^{-1}_2\cdots v^{-1}_M\iff v_M(\cdots (v_2(v_1 wu_1)u_2)\cdots )u_M=_G 1,
\]
which, by the claim above, is equivalent to $w=_H 1$.

This proves Claim 2 together with the proposition.
\end{proof}

\begin{corollary}\label{co:PCP_free}
Let $F$ be a free non-abelian group of finite rank.  Then $\GPCP(F)$ is undecidable.
\end{corollary}
\begin{proof}
It is known~\cite{Miller} that  for any natural number $n \geq 2$ there are finitely presented groups with $n$ generators and undecidable word problem.  Therefore,  $\HWP(F)$ is undecidable.  By Proposition~\ref{prop:WPtoPCP} $\GPCP(F)$ is also undecidable. \end{proof}

For a finite group presentation $P=\langle a_1, \ldots,a_k \mid r_1, \ldots, r_\ell\rangle$ denote by $N(P) = k+\ell$ the total sum of the number of generators and relators in $P$. Let $N$ be the least number $N(P)$ among all finite  presentations $P$ with undecidable word problem. In  \cite{Borisov:1969}  Borisov constructed a finitely presented group with $4$ generators and  $12$ 
relations which has undecidable word problem, so $N \leq 16$.
%Symmetrizing this presentation (adding inverses of all generators and relators  to the presentation) one gets a presentation $P^\prime$ with $N(P^\prime) = 32$. 

%{\an I do not understand why one needs $k+\ell$. $k$ seems enough.}
\begin{corollary}\label{co:twisted-free}
Let $F_n$ be a free group of rank $n \geq 32$. Then the endomorphism double twisted conjugacy problem in $F_n$ (as well as $\GPCP_n(F_n)$) is undecidable.
\end{corollary}
\begin{proof}
Let $P^\prime = \langle a_1, \ldots,a_4 \mid r_1, \ldots, r_{12}\rangle $ be the Borisov's presentation and $F_n = \langle a_1, \ldots,a_n \rangle$ a free group of rank $n\geq 32$. Claim 2 in the proof of Proposition \ref{prop:WPtoPCP} shows that the word problem in the group $H$ defined by the presentation $P^\prime$ is polynomial time reducible to $\GPCP_n(F_n)$, hence the latter one is undecidable.  Now the part 2 in Proposition \ref{pr:reduction-GPCP-2} shows that the endomorphism double twisted conjugacy problem in $F_n$ is also undecidable, as claimed. 
\end{proof}

Note that the twisted conjugacy problem is decidable in free groups~\cite{BMMV}. Together with Corollary~\ref{co:twisted-free}, this gives the following result.
\begin{corollary}\label{co:venturas_question}
Free groups of rank at least $32$ have decidable twisted conjugacy problem but undecidable endomorphism double twisted conjugacy problem.
\end{corollary}
\begin{remark}
Note that for a given group, decidability of the endomorphism double twisted conjugacy problem implies decidability of the twisted conjugacy problem, which in turn implies decidability of the conjugacy problem. It was shown in~\cite{BMV} that the converse to the latter implication is in general false. The above result~\ref{co:venturas_question} answers E.~Ventura's question whether the converse to the former implication is true.
\end{remark}

Similar results hold for free solvable groups.  Let $N_{sol}$ be the least number $N(P)$ among all finite  presentations $P$ which define a solvable group with undecidable word problem. In  \cite{Kharlampovich:1981}  Kharlampovich  constructed a finitely presented solvable group with undecidable word problem, so such number $N_{sol}$ exists. 

\begin{corollary}\label{co:PCP_free_sol}
Let $S_{m,n}$ be a free solvable non-abelian group of class $m\ge 3$ and rank  $n \geq N_{sol}$. Then the endomorphism double  twisted conjugacy problem in $S_{m,n}$ (as well as $\GPCP_n(S_{m,n})$) is undecidable.
\end{corollary}
\begin{proof}
Similar to the argument in Corollary \ref{co:twisted-free}.
\end{proof}

Observe that it immediately follows from definitions that decidability of $\PCP$ or $\GPCP$ in a finitely generated  group is inherited by all finitely generated subgroup of $G$. Therefore, the results above give a host of groups with undecidable $\GPCP$ (as well as  $\GPCP_n$).

\begin{corollary}
If a group $G$ contains a free non-abelian subgroup $F_2$ then $\GPCP(G)$  is undecidable.
\end{corollary}

Therefore $\GPCP$ is undecidable, for example, in non-elementary hyperbolic groups, non-abelian right angled Artin groups, groups with non-trivial splittings into free products with amalgamation or HNN extensions, braid groups $B_n$, non-virtually solvable linear groups, etc.

Another corollary of the results above concerns with complexity of the bounded $\GPCP$ in groups.

\begin{corollary}\label{co:NP_PCP_free}
Let $F$ be a non-abelian free group of finite rank. Then the bounded $\GPCP(F)$ is $\NP$-complete.
\end{corollary}
\begin{proof}
Let $F = F(A)$ be a free non-abelian group with a finite basis $A$. It is showed in~\cite[Corollary 1.1]{Sapir-Birget-Rips:2002}
that there exists a finitely presented group $H=\langle B\mid R\rangle$ with $\NP$-complete word problem and polynomial Dehn function $\delta_H(n)$. Passing to a subgroup of $F(A)$, we may assume that $A=B$.
One can see that in the case of a free group $G=F(A)$, $M$ in~(\ref{eq:pcp_witness}) is bounded by a
polynomial (in fact, linear) function of $|w|$ and  the number $m$ of relators in~(\ref{eq:wp}) (see~\cite[Lemma 1]{Olshanskii_Sapir:2001} for details). Note that there exists 
$m$ as above bounded by $\delta_H(|w|)$, so $M$ is bounded by some polynomial $q(|w|)$. Therefore, the map
\[w\to (w, D_R, M=q(|w|))\]
is a $\P$-time reduction of the word problem in $H$ to the bounded $\GPCP(F(A))$. It follows that the latter is $\NP$-hard and therefore
$\NP$-complete (since the word problem in $F(A)$ is $\P$-time decidable).
\end{proof}

\begin{corollary}\label{co:NP_PCP_free_sub}
If a group $G$ contains a free non-abelian subgroup $F_2$ then  the bounded $\GPCP(G)$ is $\NP$-hard.
\end{corollary}

\section{$\PCP$ in nilpotent groups}
\label{se:nilpotent}

In this section we study complexity of Post correspondence problems in nilpotent groups.

\begin{proposition}\label{prop:abelian}
There is a polynomial time algorithm that given finite presentations of groups $A$, $B$ in the class of abelian groups and a homomorphism  $\phi:A \to B$ computes a finite set of  generators of the kernel of $\phi$.
\end{proposition}
\begin{proof}
%Should be known. Otherwise,
Results of~\cite{Kannon-Bachem} provide a polynomial time algorithm to bring an integer matrix to its canonical diagonal (Smith) normal form. Since computing the canonical presentation of a finitely presented abelian group reduces by a standard argument to finding Smith form of an integer matrix (determined by relators in a given presentation), we
may find in polynomial time the canonical presentation of $B$, i.e. a direct decomposition $B = \mathbb{Z}^l \times K$, where $K$ is a finite abelian group. Once $B$ is in its canonical form, computing kernel of $\phi$ reduces to solving a system of linear equations in $Z^l$ and $K$, which can be done in polynomial time by the same results~\cite{Kannon-Bachem}.
\end{proof}

\begin{corollary}\label{co:PCP_abelian}
There is a polynomial time algorithm that given finite presentations of groups $A$, $B$ in the class of abelian groups and homomorphisms $\phi, \psi \in \Hom(A,B)$ computes a finite set of  generators of the equalizer $E(\phi,\psi)$.
\end{corollary}

\begin{proof}
Observe that a map $\xi:A \to B$ defined by $\xi(g) = \phi(g)\psi(g)^{-1}$ is a homomorphism from $A$ to $B$ and $E(\phi,\psi) = \ker\xi$. Now the result follows from Proposition \ref{prop:abelian}.
\end{proof}

One can slightly strengthen the corollaries above.
\begin{corollary}\label{co:PCP_abelian_2} Let $c$ be a fixed positive integer.
\begin{itemize}
\item [1)]  There is a polynomial time algorithm that given a finite presentation of a group $A$ (perhaps in the class of nilpotent groups of class $c$), and a finite presentation of a group $B$ in the class of abelian groups, and a homomorphism  $\phi \in \Hom(A,B)$  computes a finite set of  generators of the kernel $\ker \phi$  modulo the commutant $[A,A]$.
\item [2)] There is a polynomial time algorithm that given a finite presentation of a group $A$ (perhaps in the class of nilpotent groups of class $c$), and a finite presentation of a group $B$ in the class of abelian groups, and a homomorphism  $\phi, \psi \in \Hom(A,B)$  computes a finite set of  generators of the equalizer $E(\phi,\psi)$ modulo the commutant $[A,A]$.
\end{itemize}
\end{corollary}
\begin{proof}
Follows immediately from Proposition~\ref{prop:abelian} and Corollary~\ref{co:PCP_abelian}.
\end{proof}
%Note that Corollary~\ref{co:PCP_abelian_2} also holds in the case when the group $A$ is given by a finite presentation.

By $\gamma_{c}(G)$ we denote the $c$'s term of the lower central series of $G$. Recall that the iterated commutator of elements $g_1,\ldots, g_c$ is $[g_1,g_2,\ldots,g_c]=[\ldots [[g_1,g_2],g_3],\ldots ]$. The following lemma is well known (for example, see~\cite[Lemma 17.2.1]{Kargapolov-Merzlyakov}).

\begin{lemma}\label{le:nilpotent_basis}
Let $G$ be a group generated by elements $x_1,\ldots, x_n\in G$. Then $\gamma_c(G)$ is generated as a subgroup by $\gamma_{c+1}(G)$ and iterated commutators $[x_{i_1},\ldots, x_{i_{c}}]$.
\end{lemma}

\begin{lemma}\label{le:commutant}
Let $c_0$ be a fixed positive integer. There is a polynomial time algorithm that given a finite group presentation of a group $G$ in the class of nilpotent groups of class $\le c_0$, finds subgroup generators of $[G,G]$.
\end{lemma}
\begin{proof}
Follows from Lemma~\ref{le:nilpotent_basis} by an inductive construction since there are at most $n^{c_0+1}$ iterated commutators $[x_{i_1},\ldots, x_{i_{c}}]$, $c\le c_0$, in a group generated by $n\ge 2$ elements $x_1,\ldots,x_n$ (the case $n=1$ is obvious).
\end{proof}

\begin{theorem}\label{th:equalizer}
Let $c_0$ be a fixed positive integer. Then there is a polynomial time algorithm that given positive integers $c_H, c_G\le c_0$, finite presentations of groups $H,G$ in the classes of nilpotent groups of class $c_H$ and $c_G$, respectively, and homomorphisms $\phi, \psi \in Hom(H,G)$ computes the generating set of the equalizer $E(H,\phi,\psi)$ as a subgroup of $H$.
\end{theorem}
\begin{proof}
Let $Y$ and $Z$ be finite generating sets of $H$ and $G$, respectively. We use  induction on the nilpotency class $c=c_G$  of $G$. If $c = 1$ then $G$ is abelian and the result follows from Corollary~\ref{co:PCP_abelian_2} and~\ref{co:PCP_abelian}.

Suppose now that $c >1$ and we  are given $\phi, \psi \in \Hom(H,G)$. Consider the quotient group $\bar G = G/\gamma_c(G)$, which is a nilpotent group of class $c-1$. The homomorphisms $\phi, \psi$ induce some homomorphisms $\phi', \psi' \in \Hom(H, \bar G)$.  Observe that the size of $\phi', \psi'$ (the total length of the images $\phi(y), \psi(y), y \in Y$ as words in $Z$) is the same as of $\phi, \psi$. Also observe that $\bar G$ is described in the class of nilpotent groups of class $c-1$ by the same presentation that describes $G$ in the class of nilpotent groups of class $c$. By induction we can compute in polynomial time a finite generating set, say $L = \{h_1, \ldots, h_k\}$,  of $E'=E(H,\phi',\psi')$ as a subgroup of $H$.  By construction,  for $g\in E'$ one has $\phi(g)  = \psi(g) \mod \gamma_c(G)$, hence  a map  $\xi(g) = \phi(g)\psi(g)^{-1}$ defines a homomorphism $\xi:E' \to \gamma_c(G)$.
Obviously, $E(\phi,\psi) = \ker \xi$. Further, note that the size of $L$ is polynomial in terms of size the input, and the size of a generating set for $\gamma_c(G)$ is polynomial (of degree that depends on $c$) in terms of size of a generating set for $G$ by Lemma~\ref{le:nilpotent_basis}. Now the result follows from Corollary \ref{co:PCP_abelian_2}, item 1), since $\gamma_c(G)$ is abelian, and Lemma~\ref{le:commutant}

\end{proof}

\begin{theorem}\label{th:nilp_P} Let $c$ be a fixed positive integer.
\begin{itemize}
\item [1)] Let $G$ be a finitely generated nilpotent group of class $c$. Then for any $\phi, \psi \in \Hom(F_n,G)$ the subgroup $E(\phi,\psi) \leq F_n$ contains $\gamma_{c+1}(F_n)$ and is finitely generated modulo $\gamma_{c+1}(F_n)$.
\item [2)]  There is a polynomial time algorithm that given a positive integer $n$, a presentation of a group $G$ in the class of nilpotent groups of class $c$ and homomorphisms $\phi, \psi\in \Hom(F_n,G)$ computes a finite set of generators of $E(\phi,\psi)$ in $F_n$ modulo the subgroup $\gamma_{c+1}(F_n)$.
\end{itemize}
\end{theorem}
\begin{proof}
Let $F_n = F_n(X)$, where $X = \{x_1, \ldots,x_n\}$. Fix two homomorphisms $\phi, \psi \in \Hom(F_n,G)$. Since $G$ is nilpotent of class $c$ one has $\gamma_{c+1}(G) = 1$, so $E(\phi,\psi) \geq \gamma_{c+1}(F_n)$.  The quotient $N_{n,c} = F_n/\gamma_{c+1}(F_n)$ is a  finitely generated free nilpotent group of rank $n$ and  class $c$, hence every its subgroup, in particular the image ${\bar E}$ of $E(\phi,\psi)$, is finitely generated.  It follows that the group $E(\phi,\psi)$ is finitely generated modulo $\gamma_{c+1}(F_n)$. This proves 1). Notice, that the argument above allows one to reduce everything to the case of nilpotent groups, i.e., to consider the induced homomorphisms $\bar \phi, \bar \psi \in \Hom(N_{n,c}, G)$, instead of $\phi, \psi$, and the subgroup $\bar E$ instead of $E(\phi,\psi)$. Now the result follows from Theorem \ref{th:equalizer}.

\end{proof}

\begin{theorem} \label{th:nilpotent}
Let $G$ be a finitely generated nilpotent group.
Then $\PCP_n(G) \in \P$ for every $n \in \mathbb{N}$.
\end{theorem}

\begin{proof}
Indeed, by Theorem \ref{th:nilp_P} one can compute in $\P$-time a finite set of elements $h_1, \ldots,h_m \in F_n$ such that $E(\phi,\psi) = \langle h_1, \ldots,h_m, \gamma_{c+1}(F_n) \rangle$. Now the instance of $\PCP_n$ defined by $(\phi, \psi)$ has a non-trivial solution in $G$ if and only if there is $i$ such that $\phi(h_i) \neq 1$ in $G$. Indeed, in this case $\phi(h_i) = \psi(h_i) \neq 1$ in $G$. Otherwise, $\phi(E(\phi,\psi)) = 1$ in $G$ so there is no a non-trivial solution in $G$ to the instance of $\PCP_n$ determined by $\phi$ and $\psi$. This proves the theorem.
\end{proof}

\bibliography{main_bibliography}

\end{document}